\documentclass[12pt]{amsart}
\usepackage{amssymb,amsmath,amsthm,amsrefs}
\usepackage{tikz}
\usetikzlibrary{patterns}
\usepackage{cancel}
\usepackage{mathtools,float,stmaryrd,mathrsfs}
\usepackage[T1]{fontenc}
\usepackage{enumerate}
\usepackage{todonotes}
\usepackage{ifthen}

\usepackage{xargs}

\newcommand{\shadetheboxesPM}[1]{
    \foreach \x/\y in {#1}
    \fill[pattern color = black!75, pattern=north east lines] (\x,\y) rectangle +(1,1);
}

\newcommand{\drawthegrid}[1]{
    \draw (0.01,0.01) grid (#1+0.99,#1+0.99);
}

\newcommand{\drawverticallines}[3]{
    \foreach \x in {#2}
    \draw[line width=#3] (\x+0.01,0.01) -- (\x+0.01,#1+0.99);
}

\newcommand{\drawhorizontallines}[3]{
    \foreach \y in {#2}
    \draw[line width=#3] (0.01,\y+0.01) -- (#1+0.99,\y+0.01);
}

\newcommand{\drawtheclpattern}[1]{
    \foreach \x/\y in {#1}
    \filldraw (\x,\y) circle (6pt);
}

\newcommand{\drawclpattern}[2]{
	\foreach[count=\x] \y in {#1}
	{
		\filldraw (\x,\y) circle (#2 pt);
	}
}

\newcommand{\drawspecialbox}[1]{
    \foreach \x/\y/\z/\w/\A in {#1}
    {
        \fill[color = white!100, opacity=1, rounded corners = 1.5pt] (\x+0.125,\y+0.125) rectangle (\z-0.125,\w-0.125);
        \draw[color = black, rounded corners = 1.5pt] (\x+0.125,\y+0.125) rectangle (\z-0.125,\w-0.125);
        \fill[black] (\x/2+\z/2,\y/2+\w/2) node {\A};
    }
}

\newcommand{\drawspecialboxlarge}[1]{
    \foreach \x/\y/\z/\w/\A in {#1}
    {
        \fill[color = white!100, opacity=1, rounded corners = 1.5pt] (\x+0.125,\y+0.125) rectangle (\z-0.125,\w-0.125);
        \draw[color = black, rounded corners = 1.5pt] (\x+0.125,\y+0.125) rectangle (\z-0.125,\w-0.125);
        \fill[black] (\x/2+\z/2,\y/2+\w/2) node {\Large \A};
    }
}

\newcommand{\drawsolidshadedbox}[1]{
    \foreach \x/\y/\z/\w/\A in {#1}
    {
        \fill[color = gray!50, opacity=1, rounded corners=1.5pt] (\x+0.125,\y+0.125) rectangle (\z-0.125,\w-0.125);
        \draw[color = black, rounded corners=1.5pt] (\x+0.125,\y+0.125) rectangle (\z-0.125,\w-0.125);
        \fill[black] (\x/2+\z/2,\y/2+\w/2) node {\A};
    }
}

\newcommand{\drawlabels}[1]{
	\foreach \x/\y/\lab in {#1}
	{
		\draw (\x + 0.5,\y + 0.5) node {\lab};
	}
}

\newcommandx{\patt}[9][4={},5={},6={},7={},8={},9=4]
{
	\scalebox{#1}
	{
		\begin{tikzpicture}[baseline=(current bounding box.center)]
			\useasboundingbox (0.0,-.3) rectangle (#2+1,#2+1.3);
			\shadetheboxesPM{#4}
			\draw (0.01,0.01) grid (#2+1-0.01,#2+1-0.01);

			\drawsolidshadedbox{#6}
			\drawspecialbox{#7}
			\drawspecialboxlarge{#5}
			\drawclpattern{#3}{#9}
			\drawlabels{#8}
		\end{tikzpicture}
	}
}

\newcommandx{\cpatt}[8][4={},5={},6={},7={},8={}]
{
	\scalebox{#1}
	{
		\begin{tikzpicture}[baseline=(current bounding box.center)]
			\useasboundingbox (0.0,-.3) rectangle (#2+1,#2+1.3);
			\shadetheboxesPM{#4}
			\draw (0.01,0.01) grid (#2+1-0.01,#2+1-0.01);

			\drawsolidshadedbox{#6}
			\drawspecialbox{#7}
			\drawspecialboxlarge{#5}
			\drawclpattern{#3}{4}

			\foreach \x/\y in {#8}
			{
				\draw[line width=1] (\x,\y) circle (7 pt);
			}
		\end{tikzpicture}
	}
}

\newcommandx{\metapatt}[8][6={},7={},8={}]
{
    \scalebox{#1}
    {
        \begin{tikzpicture}[baseline=(current bounding box.center)]
					\foreach \width/\height in {#2}
					{
						\useasboundingbox (0.0,-.3) rectangle (\width+1,\height+1.3);
            \shadetheboxesPM{#6}

            \foreach \pos/\type in {#4}
            {
                \ifthenelse{\equal{\type}{v}}
                {
                    \drawverticallines{\height}{\pos}{1.7pt}
                }
                {
								    \ifthenelse{\equal{\type}{d}}
                    {
                      \draw[densely dashed] (\pos,0) -- (\pos,\height+1);
                    }
										{
											\drawhorizontallines{\width}{\pos}{1.7pt}
										}
                }
            }

            \foreach \pos/\type in {#3}
            {
                \ifthenelse{\equal{\type}{v}}
                {
                    \drawverticallines{\height}{\pos}{0.6pt}
                }
                {
										\drawhorizontallines{\width}{\pos}{0.6pt}
                }
            }

            \drawsolidshadedbox{#8}
            \drawspecialbox{#7}

            \foreach \x/\y/\type in {#5}
            {
                \ifthenelse{\equal{\type}{a}}
                {
                    \draw (\x,\y) circle (6pt);
                    \filldraw (\x,\y) circle (3pt);
                }
                {
                    \filldraw (\x,\y) circle (4pt);
                }
            }
					}
        \end{tikzpicture}
    }
}

\newcommandx{\dpatt}[9][6={},7={},8={},9={}]
{
    \scalebox{#1}
    {
        \begin{tikzpicture}[baseline=(current bounding box.center)]
					\foreach \width/\height in {#2}
					{
						\useasboundingbox (0.0,-.3) rectangle (\width+1,\height+1.3);
            \shadetheboxesPM{#6}

            \foreach \pos/\type in {#4}
            {
                \ifthenelse{\equal{\type}{v}}
                {
                    \drawverticallines{\height}{\pos}{1.7pt}
                }
                {
								    \ifthenelse{\equal{\type}{d}}
                    {
                      \draw[densely dashed] (\pos,0) -- (\pos,\height+1);
                    }
										{
											\drawhorizontallines{\width}{\pos}{1.7pt}
										}
                }
            }

            \foreach \pos/\type in {#3}
            {
                \ifthenelse{\equal{\type}{v}}
                {
                    \drawverticallines{\height}{\pos}{0.6pt}
                }
                {
										\drawhorizontallines{\width}{\pos}{0.6pt}
                }
            }

            \drawsolidshadedbox{#8}
            \drawspecialbox{#7}

            \foreach \x/\y/\type in {#5}
            {
                \ifthenelse{\equal{\type}{a}}
                {
                    \draw9 (\x,\y) circle (6pt);
                    \filldraw (\x,\y) circle (3pt);
                }
                {
                    \filldraw (\x,\y) circle (4pt);
                }
            }

						\drawlabels{#9}
					}
        \end{tikzpicture}
    }
}

\newcommand{\mpattern}[4]{										
  \raisebox{0.6ex}{
  \begin{tikzpicture}[scale=0.35, baseline=(current bounding box.center), #1]
  	\useasboundingbox (0.0,-0.1) rectangle (#2+1.4,#2+1.1);

    \shadetheboxesPM{#4}

    \drawthegrid{#2}

    \drawtheclpattern{#3}

  \end{tikzpicture}}
}

\pgfmathsetmacro{\patttablescale}{1.05}
\pgfmathsetmacro{\pattdispscale}{0.80}
\pgfmathsetmacro{\patttextscale}{0.6}

\usetikzlibrary{arrows,shapes,backgrounds,decorations,positioning,patterns}
\colorlet{lightgray}{black!15}

\theoremstyle{plain}
\newtheorem{theorem}{Theorem}[section]
\newtheorem{proposition}[theorem]{Proposition}
\newtheorem{corollary}[theorem]{Corollary}
\newtheorem{lemma}[theorem]{Lemma}

\theoremstyle{definition}
\newtheorem{definition}[theorem]{Definition}
\newtheorem{remark}[theorem]{Remark}
\newtheorem{example}[theorem]{Example}

\newcommand{\mf}[1]{\mbox{$\mathfrak #1$}}

\newcommand{\ms}[1]{\mbox{$\mathscr #1$}}

\newcommand{\Grid}{\mathrm{Grid}}
\newcommand{\Av}{\mathrm{Av}}
\newcommand{\st}{\mathrm{st}}
\newcommand{\id}{\mathrm{id}}

\newcommand*{\pw}{3 pt}

\newcommand*{\lw}{1 pt}

\newcommand{\thegridext}[2][2]{
    \pgfmathsetmacro\max{#1+#2}
    \pgfmathsetmacro\maxes{#1-2}
    \foreach \i in {1,...,#2}{
        \draw[line width=\lw] (\i-1, \max-2) -- (\i-1, #2-\i) -- (\max-1, #2-\i);
    }
    \ifnum#1>2
    \foreach \i in {1,...,\maxes}{
        \draw[line width=\lw] (0, \max-\i-1) -- (#2+\i-1, \max-\i-1) -- (#2+\i-1, 0);
    }
    \fi
        \draw[line width=\lw] (0, #2) -- (\max-1, #2) -- (\max-1, 0);

    \draw[line width=\lw] (\max-2, #2) -- (\max-2, 0);
}

\newcommand*{\radius}{2 cm}

\newcommand{\thenodesA}[2][(0,0)]{
    \begin{scope}[shift={#1}]
    \tikzstyle{mynode}=[circle, draw, thin,fill=gray!20, scale=0.8]
    \foreach \i in {1,...,#2}{
        \foreach \j in {\i,...,#2}{
            \filldraw (\j-0.5,-\i-0.5) circle (\pw);
            \node[mynode] (\i \j) at (\j-0.5,-\i-0.5) {\tiny$\i\j$};
        }
    }
    \ifthenelse{#2>2}{
    \pgfmathtruncatemacro\rowstop{#2-2}
    \pgfmathtruncatemacro\columnstop{#2-1}
    \foreach \i in {1,...,\rowstop}{
        \pgfmathtruncatemacro\ii{\i+1}
        \foreach \j in {\ii,...,\columnstop}{
            \pgfmathtruncatemacro\jj{\j+1}
            \foreach \k in {\ii,...,\j}{
                \foreach \m in {\jj,...,#2}{
                    \path[thick] (\i \j) edge (\k \m);
                }
            }
        }
    }
    }{}
    \end{scope}
}

\newcommand{\thenodesAsmall}[2][(0,0)]{
    \begin{scope}[shift={#1}]
    \tikzstyle{mynode}=[circle, draw, thin,fill=gray!20, scale=0.4]
    \foreach \i in {1,...,#2}{
        \foreach \j in {\i,...,#2}{
            \filldraw (\j-0.5,-\i-0.5) circle (\pw);
            \node[mynode] (\i \j) at (\j-0.5,-\i-0.5) {\tiny$\i\j$};
        }
    }
    \ifthenelse{#2>2}{
    \pgfmathtruncatemacro\rowstop{#2-2}
    \pgfmathtruncatemacro\columnstop{#2-1}
    \foreach \i in {1,...,\rowstop}{
        \pgfmathtruncatemacro\ii{\i+1}
        \foreach \j in {\ii,...,\columnstop}{
            \pgfmathtruncatemacro\jj{\j+1}
            \foreach \k in {\ii,...,\j}{
                \foreach \m in {\jj,...,#2}{
                    \path[thick] (\i \j) edge (\k \m);
                }
            }
        }
    }
    }{}
    \end{scope}
}

\newcommand{\thesubnodesA}[2]{
    \tikzstyle{mynode}=[circle, draw, thin,fill=gray!20, scale=0.8]
    \foreach \i in {1,...,#1}{
        \foreach \j in {#1,...,#2}{
            \filldraw (\j-0.5,-\i-0.5) circle (\pw);
            \node[mynode] (\i \j) at (\j-0.5,-\i-0.5) {\tiny$\i\j$};
        }
    }
    \ifthenelse{#2>2}{
    \pgfmathtruncatemacro\rowstop{#1-1}
    \pgfmathtruncatemacro\columnstop{#2-1}
    \foreach \i in {1,...,\rowstop}{
        \pgfmathtruncatemacro\ii{\i+1}
        \foreach \j in {#1,...,\columnstop}{
            \pgfmathtruncatemacro\jj{\j+1}
            \foreach \k in {\ii,...,#1}{
                \foreach \m in {\jj,...,#2}{
                    \path[thick] (\i \j) edge (\k \m);
                }
            }
        }
    }
    }{}
}

\newcommand{\thenodesB}[2][(0,0)]{
    \begin{scope}[shift={#1}]
    \tikzstyle{mynode}=[circle, draw, thin,fill=gray!20, scale=0.8]
    \foreach \i in {1,...,#2}{
        \foreach \j in {\i,...,#2}{
            \filldraw (\j-0.5,-\i-0.5) circle (\pw);
            \node[mynode] (\i \j) at (\j-0.5,-\i-0.5) {\tiny$\i\j$};
        }
    }
    \ifthenelse{#2>2}{
    \pgfmathtruncatemacro\rowstop{#2-2}
    \foreach \i in {1,...,\rowstop}{
        \pgfmathtruncatemacro\ii{\i+1}
        \pgfmathtruncatemacro\iii{\i+2}
        \foreach \j in {\iii,...,#2}{
            \pgfmathtruncatemacro\jj{\j+1}
            \pgfmathtruncatemacro\ju{\j-1}
            \foreach \k in {\ii,...,\ju}{
                \foreach \m in {\k,...,\ju}{
                    \path[thick] (\i \j) edge (\k \m);
                }
            }
        }
    }
    }{}
\end{scope}
}

\newcommand{\thenodesBsmall}[2][(0,0)]{
    \begin{scope}[shift={#1}]
    \tikzstyle{mynode}=[circle, draw, thin,fill=gray!20, scale=0.4]
    \foreach \i in {1,...,#2}{
        \foreach \j in {\i,...,#2}{
            \filldraw (\j-0.5,-\i-0.5) circle (\pw);
            \node[mynode] (\i \j) at (\j-0.5,-\i-0.5) {\tiny$\i\j$};
        }
    }
    \ifthenelse{#2>2}{
    \pgfmathtruncatemacro\rowstop{#2-2}
    \foreach \i in {1,...,\rowstop}{
        \pgfmathtruncatemacro\ii{\i+1}
        \pgfmathtruncatemacro\iii{\i+2}
        \foreach \j in {\iii,...,#2}{
            \pgfmathtruncatemacro\jj{\j+1}
            \pgfmathtruncatemacro\ju{\j-1}
            \foreach \k in {\ii,...,\ju}{
                \foreach \m in {\k,...,\ju}{
                    \path[thick] (\i \j) edge (\k \m);
                }
            }
        }
    }
    }{}
\end{scope}
}

\newcommand{\thesubnodesB}[2]{
    \tikzstyle{mynode}=[circle, draw, thin,fill=gray!20, scale=0.8]
    \foreach \i in {1,...,#1}{
        \foreach \j in {#1,...,#2}{
            \filldraw (\j-0.5,-\i-0.5) circle (\pw);
            \node[mynode] (\i \j) at (\j-0.5,-\i-0.5) {\tiny$\i\j$};
        }
    }
    \ifthenelse{#2>2}{
    \pgfmathtruncatemacro\rowstop{#1-1}
    \pgfmathtruncatemacro\columnbegin{#1+1}
    \foreach \i in {1,...,\rowstop}{
        \pgfmathtruncatemacro\ii{\i+1}
        \foreach \j in {\columnbegin,...,#2}{
            \pgfmathtruncatemacro\columnstop{\j-1}
            \foreach \k in {\ii,...,#1}{
                \foreach \m in {#1,...,\columnstop}{
                    \path[thick] (\i \j) edge (\k \m);
                }
            }
        }
    }
    }{}
}

\newcommand{\completegraph}[2][(0,0)]{
    \begin{scope}[shift={#1}]
    \foreach \s in {1,...,#2}{
        \node[draw, circle] (\s) at ({-360/#2 * (\s - 1)}:\radius) {\tiny$\s$}; 
    }
    \foreach \s in {1,...,#2}{
        \foreach \t in {\s,...,#2}{
            \pgfmathsetmacro\tt{\t-1}
            {\ifnum\s=\t
            {}
            \else
            \draw (\t) -- (\s); 
            \fi
            }
        }
    }
    \end{scope}
}
\newcommand{\setpartitions}[2][(0,0)]{
    \begin{scope}[shift={#1}]
        \foreach \s in {1,...,#2}{
            \node (\s) at (\s-1,0) {\tiny$\s$};
        }
        \foreach \s in {1,...,#2}{
            \foreach \t in {\s,...,#2}{
            {\ifnum\s=\t
            {}
            \else
            \draw (\s) to [out=90, in=90] (\t);
            \fi
            }
            }
        }
    \end{scope}
}

\title[Occurrence graphs of patterns in permutations]
{Occurrence graphs of patterns in permutations}

\author[Bjarni Jens Kristinsson]{Bjarni Jens Kristinsson$^{\star}$}
\address{Department of Mathematics, University of Iceland, Reykjavik, Iceland}
\email{bjk17@hi.is}
\author[Henning Ulfarsson]{Henning Ulfarsson$^{\star}$}
\address{School of Computer Science, Reykjavik University, Reykjavik, Iceland}
\email{henningu@ru.is}

\thanks{$^{\star}$ Research partially supported by grant 141761-051 from the
Icelandic Research Fund.}


\begin{document}

\begin{abstract}
    We define the \emph{occurrence graph} $G_p(\pi$) of a pattern $p$ in a
    permutation $\pi$ as the graph with the occurrences of $p$ in $\pi$ as
    vertices and edges between the vertices if the occurrences differ by exactly
    one element. We then study properties of these graphs. The main theorem in
    this paper is that every \emph{hereditary property} of graphs gives rise to
    a \emph {permutation class}.

\end{abstract}

\maketitle
\thispagestyle{empty}

\section{Introduction}
\label{sec:introduction}
\pagenumbering{arabic}
\setcounter{page}{1}

We define the \emph{occurrence graph} $G_p(\pi$) of a pattern $p$ in a
permutation $\pi$ as the graph where each vertex represents an occurrences of
$p$ in $\pi$. Vertices share an edge if the occurrences they represent differ by
exactly one element. We study properties of these graphs and show that every
\emph{hereditary property} of graphs gives rise to a \emph{permutation class}.

The motivation for defining these graphs comes from the algorithm discussed
in the proof of the \emph{Simultaneous Shading Lemma} by Claesson, Tenner and
Ulfarsson in~\cite{simshadinglemma}. The steps in that algorithm can
be thought of as constructing a path in an occurrence graph, terminating at a
desirable occurrence of a pattern.

\section{Basic definitions}
\label{sec:basic-definitions}

In this article we will be working with permutations and undirected, simple
graphs. The reader does not need to have prior knowledge of either as we will
define both.

\begin{definition}\label{defn:graph}
  A \emph{graph} is an ordered pair $G=(V,E)$ where $V$ is a set of
  \emph{vertices} and $E$ is a set of two element subsets of $V$. The elements
  $\{u,v\} \in E$ are called \emph{edges} and connect the vertices. Two vertices
  $u$ and $v$ are \emph{neighbors} if $\{u,v\} \in E$. The \emph{degree} of a
  vertex $v$ is the number of neighbors it has. A graph $G' = (V',E')$ is a
  \emph{subgraph} of $G$ if $V' \subseteq V$ and $E' \subseteq \bigl\{ \{u,v\}
  \in E \colon u,v \in V' \bigr\}$.
\end{definition}

The reader might have noticed that our definition of a graph excludes those with
loops and multiple edges between vertices. We often write $uv$ as shorthand for
$\{u,v\}$ and in case of ambiguity we use $V(G)$ and $E(G)$ instead of $V$ and
$E$.

\begin{definition}\label{defn:induced-and-isomorphic-graphs}
  Let $V'$ be a subset of $V$. The \emph{induced subgraph} $G[V']$
  is a subgraph of $G$ with vertex set $V'$ and edges $\{ uv \in E
  \colon u,v \in V' \}$.

  Two graphs $G$ and $H$ are \emph{isomorphic} if there exist a
  bijection from $V(G)$ to $V(H)$ such that two vertices in $G$ are
  neighbors if and only if the corresponding vertices (according to
  the bijection) in $H$ are neighbors. We denote this with $G \cong
  H$.
\end{definition}

We let $\llbracket 1,n \rrbracket$ denote the integer interval $\{1,
\ldots, n\}$.

\begin{definition}\label{defn:permutation}
  A \emph{permutation of length $n$} is a bijective function $\sigma
  \colon \llbracket 1,n \rrbracket \to \llbracket 1,n \rrbracket$.
  We denote the permutation with $\sigma =
  \sigma(1)\sigma(2)\cdots\sigma(n)$. The permutation $\id_n = 12
  \cdots n$ is the \emph{identity permutation} of length $n$.
\end{definition}

The \emph{set of permutations} of length $n$ is denoted by $\mf{S}_n$.
The set of all permutations is $\mf{S} = \cup_{n=0}^{+\infty}
\mf{S}_n$. Note that $\mf{S}_0 = \{ \mathscr{E} \}$, where
$\mathscr{E}$ is the empty permutation, and $\mf{S}_1 = \{ 1 \}$.
There are $n!$ permutations of length $n$.

\begin{definition}\label{defn:grid-plot}
  A \emph{grid plot} or \emph{grid representation} of a permutation
  $\pi\in\mf{S}_n$ is the subset $\Grid(\pi) =
  \bigl\{\left(i,\pi(i)\right) \colon i \in \llbracket 1,n
  \rrbracket \bigr\}$ of the Cartesian product $\llbracket 1,n
  \rrbracket^2 = \llbracket 1,n \rrbracket \times \llbracket 1,n
  \rrbracket$.
\end{definition}

\begin{example}\label{ex:grid-plot}
  Let $\pi = 42135$. The grid representation of $\pi$ is

  \begin{figure}[htbp]
    \begin{gather*}
    \Grid(42135) \;=\;
      \begin{tikzpicture}[scale=.4, baseline={([yshift=-3pt]current bounding box.center)}]
        \def \n {5}
        \foreach \x in {1,...,\n} {
          \draw[gray] (0,\x) -- (\n+1,\x);
          \draw[gray] (\x,0) -- (\x,\n+1);
        }
        \foreach \x in {(1,4),(2,2),(3,1),(4,3),(5,5)} {\fill[black] \x circle (5pt);}
      \end{tikzpicture}
      \end{gather*}
  \end{figure}
\end{example}

The central definition in the theory of permutation patterns is how permutations
lie inside other (larger) permutations. Before we define that
precisely we need a preliminary definition:

\begin{definition}\label{defn:pattern-standardisation}
  Let $a_1, \ldots, a_k$ be distinct integers. The \emph
  {standardisation} of the string $a_1 \cdots a_k$ is the
  permutation $\sigma\in\mf{S}_k$ such that $a_1 \cdots a_k$ is
  order isomorphic to $\sigma(1) \cdots \sigma(k)$. In other words,
  for every $i \neq j$ we have $a_i < a_j$ if and only if
  $\sigma(i)<\sigma(j)$. We denote this with $\st(a_1 \cdots a_k) =
  \sigma$.
\end{definition}

For example $\st(253)=132$ and $\st(132)=132$.

\begin{definition}\label{defn:pattern-containment-and-occurrences}
  Let $p$ be a permutation of length $k$. We say that the permutation
  $\pi\in\mf{S}_n$ \emph{contains} $p$ if there exist indices $1 \le i_1 <
  \cdots < i_k \le n$ such that $\st\bigl(\pi(i_1) \cdots \pi(i_k)\bigr) = p$.
  The subsequence $\pi(i_1) \cdots \pi(i_k)$ is an \emph{occurrence} of $p$ in
  $\pi$ with the \emph{index set} $\{i_1, \ldots, i_k\}$. The increasing
  sequence $i_1 \cdots i_k$ will be used to denote the order preserving
  injection $i \colon \llbracket 1,k \rrbracket \rightarrow \llbracket 1,n
  \rrbracket, j \mapsto i_j$ which we call the \emph{index injection} of $p$
  into $\pi$ for this particular occurrence.

  The set of all index sets of $p$ in $\pi$ is the \emph{occurrence
  set} of $p$ in $\pi$, denoted with $V_p(\pi)$. If $\pi$ does not
  contain $p$, then $\pi$ \emph{avoids} $p$. In this context the
  permutation $p$ is called a (\emph{classical permutation}) \emph
  {pattern}.
\end{definition}

Unless otherwise stated, we write the index set $\{ i_1, \ldots, i_n
\}$ in ordered form, i.e.,~such that $i_1 < \cdots < i_n$, in
accordance with how we write the index injection.

The set of all permutations that avoid $p$ is $\Av(p)$. More
generally for a set of patterns $M$ we define \[ \Av(M) = \bigcap_{p
\in M} \Av(p) .\]

\begin{example}\label{ex:pattern-containment-and-occurrences}
  The permutation $42135$ contains five occurrences of the pattern
  $213$, namely $425$, $415$, $435$, $213$ and $215$. The occurrence
  set is
  \[
    V_{213}(42135) = \bigl\{ \{1,2,5\}, \{1,3,5\}, \{1,4,5\},
                             \{2,3,4\}, \{2,3,5\} \bigr\}.
  \]
  The permutation $42135$ avoids the pattern $132$.
\end{example}

\section{Occurrence graphs}
\label{sec:occurrence-graphs}

We now formally define occurrence graphs.

\begin{definition}\label{defn:occurrence-graph}
  For a pattern $p$ of length $k$ and a permutation $\pi$ we define the
  \emph{occurrence graph} $G_p(\pi)$ of $p$ in $\pi$ as follows:
  \begin{itemize}
    \item The set of vertices is $V_p(\pi)$, the occurrence set of
      $p$ in $\pi$.
    \item $uv$ is an edge in $G_p(\pi)$ if the vertices $u = \{u_1,
      \ldots, u_k\}$ and $v = \{v_1, \ldots, v_k\}$ in $V_p(\pi)$
      differ by exactly one element, i.e.,~if \[\left|u \setminus
      v\right| = \left|v \setminus u\right| = 1 .\]
  \end{itemize}
\end{definition}

\begin{example}\label{ex:occurrence-graph}
  In Example~\ref{ex:pattern-containment-and-occurrences} we derived
  the occurrence set $V_{213}(42135)$. We compute the edges of
  $G_{213}(42135)$ by comparing the vertices two at a time to see if
  the sets differ by exactly one element. The graph is shown in
  Figure~\ref{fig:graph-of-G_213(42135)}.

  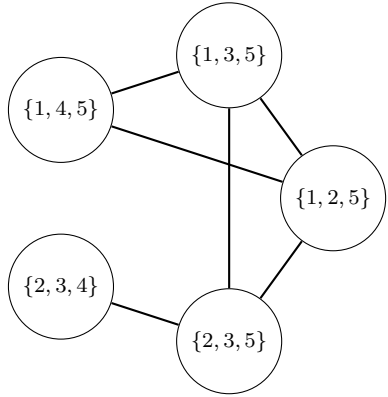
\begin{figure}[htbp]
    \begin{gather*}
    \begin{tikzpicture}
    \tikzstyle{vertex}=[draw,circle]
    \tikzstyle{edge} = [draw,thick,-,black]
    \def \radius {2cm}
    \node[vertex] (v125) at ({  0}:\radius) {\tiny{$\{1,2,5\}$}};
    \node[vertex] (v135) at ({ 72}:\radius) {\tiny{$\{1,3,5\}$}};
    \node[vertex] (v145) at ({144}:\radius) {\tiny{$\{1,4,5\}$}};
    \node[vertex] (v234) at ({216}:\radius) {\tiny{$\{2,3,4\}$}};
    \node[vertex] (v235) at ({288}:\radius) {\tiny{$\{2,3,5\}$}};
    \draw[edge] (v125) -- (v135) -- (v145) -- (v125) -- (v235);
    \draw[edge] (v135) -- (v235) -- (v234);
    \end{tikzpicture}
  \end{gather*}
  \caption{The occurrence graph $G_{213}(42135)$}
  \label{fig:graph-of-G_213(42135)}
  \end{figure}
\end{example}

\begin{remark}\label{remark:occurrence-graph}
    For a permutation $\pi$ of length $n$ the graph $G_{\mathscr{E}}(\pi)$ is a
    graph with one vertex and no edges and $G_1(\pi)$ is a clique on $n$
    vertices.
\end{remark}

Following the definition of these graphs there are several natural questions
that arise. For example, for a fixed pattern $p$, which occurrence graphs
$G_p(\pi)$ satisfy a given graph property, such as being connected or being a
tree? Before we answer questions of this sort we consider a simpler question:
What can be said about the graph $G_{12}(\id_n)$?

\section{The pattern $p=12$ and the identity permutation}
\label{sec:the-pattern-p-12-and-the-identity-permutation}

In this section we only consider the pattern $p=12$ and let $n\geq 2$. For this
choice of $p$ and a fixed $n$ the identity permutation $\pi=1\cdots n$ contains
the most occurrences of $p$. Indeed, every set $\{i,j\}$ with $i \neq j$ is an
index set of $p$ in $\pi$. We can choose this pair in \[{n \choose 2} = {n(n-1)
\over 2} \] different ways. Therefore, this is the size of the vertex set of
$G=G_p(\pi)$.

Every vertex $u=\{i,j\}$ in $G$ is connected to $n-2$ vertices $v=\{i,j'\}$, $j'
\neq j$, and $n-2$ vertices $w=\{i',j\}$, $i' \neq i$. Thus, the degree of every
vertex in $G$ is $2(n-2)$. By summing this over the set of vertices and dividing
by two we get the number of edges in $G$: \[|E(G)| = {n(n-1)(n-2) \over 2} = 3
{n \choose 3}.\]

A triangle in $G$ consists of three vertices $u,v,w$ with edges $uv,vw,wu$. If
$u = \{ i,j \}$ (not neccessarily in ordered form) then we can assume $v$ is $\{
j,k \}$. For this triplet to be a triangle $w$ must connect to both $u$ and $v$,
and therefore $w$ must either be the index set $\{ i,k \}$ or $\{ j,j' \}$ where
$j' \neq i,k$. In the first case, we just need to choose three indices $i,j,k$.
In the second case we start by choosing the common index $k$ and then we choose
the remaining indices. Thus the number of triangles in $G$ is
    \[
        {n \choose 3} + n {n-1 \choose 3} = (n-2){n \choose 3}.
    \]

\begin{example}\label{ex:p12n12345}
  The graph $G_{12}(12345)$ is pictured in Figure~\ref{fig:graph-of-G_12(12345)}.
  It has $10$ vertices, $30$ edges, and $30$ triangles. It also has $5$
  subgraphs isomorphic to $K_4$, one of them highlighted with bolder gray edges
  and gray vertices.

  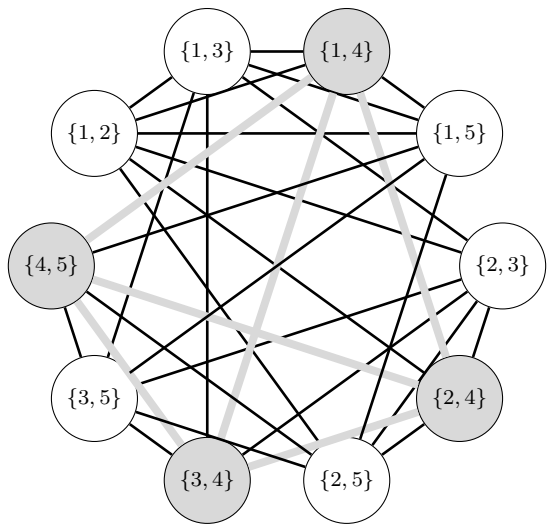
\begin{figure}[htbp]
    \begin{gather*}
      \begin{tikzpicture}
      \tikzstyle{vertex}=[draw,circle]
      \tikzstyle{gertex}=[draw,circle,fill=lightgray]
      \tikzstyle{edge1} = [line width=1pt,opacity=1.0,black]
      \tikzstyle{edge2} = [line width=1pt,opacity=1.0,black]
      \tikzstyle{edge3} = [line width=1pt,opacity=1.0,black]
      \tikzstyle{edge4} = [line width=3pt,opacity=1.0,lightgray]
      \tikzstyle{edge5} = [line width=1pt,opacity=1.0,black]
      \def \radius {3cm}
      \def \s {36} 
      \node[vertex] (v12) at ({ 4*\s}:\radius) {\tiny{$\{1,2\}$}};
      \node[vertex] (v13) at ({ 3*\s}:\radius) {\tiny{$\{1,3\}$}};
      \node[gertex] (v14) at ({ 2*\s}:\radius) {\tiny{$\{1,4\}$}};
      \node[vertex] (v15) at ({ 1*\s}:\radius) {\tiny{$\{1,5\}$}};
      \node[vertex] (v23) at ({    0}:\radius) {\tiny{$\{2,3\}$}};
      \node[gertex] (v24) at ({-1*\s}:\radius) {\tiny{$\{2,4\}$}};
      \node[vertex] (v25) at ({-2*\s}:\radius) {\tiny{$\{2,5\}$}};
      \node[gertex] (v34) at ({-3*\s}:\radius) {\tiny{$\{3,4\}$}};
      \node[vertex] (v35) at ({-4*\s}:\radius) {\tiny{$\{3,5\}$}};
      \node[gertex] (v45) at ({-5*\s}:\radius) {\tiny{$\{4,5\}$}};
      \draw[edge1] (v12) -- (v13) -- (v14) -- (v15) -- (v13);
      \draw[edge1] (v15) -- (v12) -- (v14);
      \draw[edge2] (v12) -- (v23) -- (v24) -- (v25) -- (v23);
      \draw[edge2] (v25) -- (v12) -- (v24);
      \draw[edge3] (v13) -- (v23) -- (v34) -- (v35) -- (v23);
      \draw[edge3] (v35) -- (v13) -- (v34);
      \draw[edge4] (v14) -- (v24) -- (v34) -- (v45) -- (v24);
      \draw[edge4] (v45) -- (v14) -- (v34);
      \draw[edge5] (v15) -- (v25) -- (v35) -- (v45) -- (v25);
      \draw[edge5] (v45) -- (v15) -- (v35);
    \end{tikzpicture}
  \end{gather*}
  \caption{The graph $G_{12}(12345)$}\label{fig:graph-of-G_12(12345)}
  \end{figure}
\end{example}

The following proposition generalizes the observations above to larger cliques.

\begin{proposition}\label{prop:nr-of-cliques}
  For $n>0$, the number of cliques of size $k>3$ in $G_{12}(\id_n)$
  is \[(k+1){n \choose k+1} = n {n-1 \choose k} .\]
\end{proposition}

\begin{proof}
    The vertices $(a_1,b_1), (a_2,b_2), \dots, (a_k,b_k)$ in a clique of size
    $k > 3$ must have a common index, say $\ell = a_1 = a_2 = \dots = a_k$,
    without loss of generality. The remaining indices $b_1, b_2, \dots, b_k$ can
    chosen as any subset of the other $n-1$ indices. This explains the right
    hand side of the equation in the proposition.
\end{proof}

\section{Hereditary properties of graphs}
\label{sec:hereditary}

Intuitively one would think that if a pattern $p$ is contained
inside a larger pattern $q$, that either of the occurrence graphs
$G_p(\pi)$ and $G_q(\pi)$ (for any permutation $\pi$) would be
contained inside the other. But this is not the case as the
following examples demonstrate.

\begin{example}\begin{enumerate}
  \item Let $p=12$, $q=231$ and $\pi=3421$. The occurrence sets are
    $V_p(\pi) = \bigl\{ \{1,2\} \bigr\}$ and $V_q(\pi) = \bigl\{
    \{1,2,3\}, \{1,2,4\} \bigr\}$. The cardinality of the set
    $V_p(\pi)$ is smaller then the cardinality of $V_q(\pi)$.

  \item If on the other hand $p=12$, $q=123$ and $\pi=123$ then the
    occurrence sets are $V_p(\pi) = \bigl\{ \{1,2\}, \{1,3\}, \{2,3\}
    \bigr\}$ and $V_q(\pi) = \bigl\{ \{1,2,3\} \bigr\}$. The relative
    size of the occurrence sets are now changed.
\end{enumerate}
\end{example}

However, for a fixed pattern $p$, we obtain an inclusion of occurrence graph in
Proposition~\ref{prop:permutation-subgraphs}. First we need a lemma.

\begin{lemma}\label{lem:occurrence-subset}
  Let $\pi$ and $\sigma$ be two permutations. For every occurrence
  of $\pi$ in $\sigma$ the index injection induces an injection
  $\Phi_p\colon V_p(\pi) \rightarrow V_p(\sigma)$, for all patterns
  $p$.
\end{lemma}

\begin{proof}
  Let $p$, $\pi$, $\sigma$ be permutations of length $l$, $m$, $n$
  respectively. Every $v = \{ i_1, \ldots, i_l \} \in V_p(\pi)$ is
  an index set of $p$ in $\pi$ with index injection $i$. Let $j$ be
  an index injection for an index set $\{ j_1, \ldots, j_m \}$ of
  $\pi$ in $\sigma$. It's easy to see that $u = \{ j_{i_1}, \ldots,
  j_{i_l} \}$ is an index set of $p$ in $\sigma$ because $j \circ i$
  is an index injection of $p$ into $\sigma$. Define $\Phi_p(v) = u$.
\end{proof}

\begin{example}\label{ex:occurrence-subset}
  Let $p=12$, $\pi=132$ and $\sigma=24153$. There are three occurrences of $\pi$
  in $\sigma$: $243$, $253$ and $153$ with respective index injections $125$,
  $145$ and $345$.

  For a given index injection, say $i=345$, we obtain the injection $\Phi_p$ by
  mapping every $\{v_1, v_2\} \in V_p(\pi)$ to $\{i_{v_1}, i_{v_2} \} \in
  V_p(\sigma)$. We calculate that $\Phi_p$ maps $\{1,2\}$ to $\{i_1,i_2\} =
  \{3,4\}$ and $\{1,3\}$ to $\{i_1,i_3\} = \{3,5\}$, see
  Figure~\ref{fig:occurrence-subset-example}.

  \begin{figure}[htbp]\begin{gather*}
    \begin{tikzpicture}[scale=.4, baseline={([yshift=-3pt]current bounding box.center)}]
        \foreach \x in {1,...,5} {
          \foreach \y in {1,...,5} {
            \draw[gray] (0,\y) -- (6,\y);
            \draw[gray] (\x,0) -- (\x,6);
          }
        }
        \foreach \x in {(1,2),(2,4),(3,1),(4,5),(5,3)} {\fill[black] \x circle (5pt);}
        \foreach \x in {            (3,1),(4,5),(5,3)} {\draw[gray]  \x circle (8pt);}
        \node[draw, diamond, minimum size=9pt] at (3,1){};
        \node[draw, diamond, minimum size=9pt] at (5,3){};
    \end{tikzpicture}
  \end{gather*}
    \caption{The occurrence of $\pi$ in $\sigma$ that is defined by
      the index injection $i=345$ is highlighted with gray circles.
      The occurrence set $\left\{ 1,3 \right\}$ of $p$ in $\pi$ is
      mapped with the injection $\Phi_p$, induced by $i$, to the
      index set $\left\{ 3,5 \right\}$ of $p$ in $\sigma$,
      highlighted with black diamonds}
    \label{fig:occurrence-subset-example}
  \end{figure}
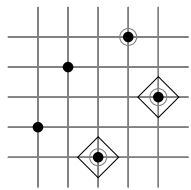
\end{example}

\begin{proposition}\label{prop:permutation-subgraphs}
  Let $p$ be a pattern and $\pi$, $\sigma$ be two permutations. For every
  occurrence of $\pi$ in $\sigma$ the index injection induces an isomorphism of
  the occurrence graph $G_p(\pi)$ with a subgraph of $G_p(\sigma)$.
\end{proposition}

\begin{proof}
  From Lemma~\ref{lem:occurrence-subset} we have the injection $\Phi_p \colon
  V_p(\pi) \rightarrow V_p(\sigma)$. We need to show that for every $uv \in
  E\left(G_p(\pi)\right)$ that $\Phi_p(u)\Phi_p(v) \in
  E\left(G_p(\sigma)\right)$. Let $uv$ be an edge in $G_p(\pi)$, where $u =
  \left\{ u_1, \ldots, u_l \right\}$ and $v = \left\{ v_1, \ldots, v_l
  \right\}$. For every index injection $j$ of $\pi$ into $\sigma$, the vertices
  $u,v$ map to $\Phi_p(u) = \left\{ j(u_1), \ldots, j(u_l) \right\}$, $\Phi_p(v) =
  \left\{ j(v_1), \ldots, j(v_l) \right\}$ respectively. Since $j$ is an
  injection there exists an edge between these two vertices in $G_p(\sigma)$.
  \end{proof}

\begin{example}\label{ex:permutation-subset}
  We will continue with Example~\ref{ex:occurrence-subset} and show
  how the index injection $i=345$ define subgraph of
  $G_p(\sigma)$ which is isomorphic to $G_p(\pi)$. The occurrence
  graph of $p$ in $\pi$ is a graph on two vertices $\{1,2\}$ and
  $\{1,3\}$ with an edge between them. The occurrence graph
  $G_p(\sigma)$ with the highlighted subgraph induced by $i$ is shown in
  Figure~\ref{fig:permutation-subset-example}.

  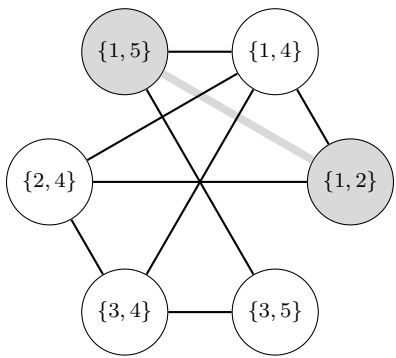
\begin{figure}[htbp]
    \begin{tikzpicture}
    \tikzstyle{vertex}=[draw,circle]
    \tikzstyle{gertex}=[draw,circle,fill=lightgray]
    \tikzstyle{edge} = [draw,thick,-,black]
    \tikzstyle{wedge}= [line width=3pt,opacity=1.0,lightgray]
    \def \radius {2cm}
    \node[gertex] (v12) at ({  0}:\radius) {\tiny{$\{1,2\}$}};
    \node[vertex] (v14) at ({ 60}:\radius) {\tiny{$\{1,4\}$}};
    \node[gertex] (v15) at ({120}:\radius) {\tiny{$\{1,5\}$}};
    \node[vertex] (v24) at ({180}:\radius) {\tiny{$\{2,4\}$}};
    \node[vertex] (v34) at ({240}:\radius) {\tiny{$\{3,4\}$}};
    \node[vertex] (v35) at ({300}:\radius) {\tiny{$\{3,5\}$}};
    \draw[edge] (v12) -- (v14) -- (v15);
    \draw[wedge] (v15) -- (v12);
    \draw[edge] (v12) -- (v24) -- (v14) -- (v34) -- (v24);
    \draw[edge] (v15) -- (v35) -- (v34);
    \end{tikzpicture}
  \caption{The graph $G_{12}(24153)$ with a highlighted subgraph
  isomorphic to $G_{12}(132)$}
  \label{fig:permutation-subset-example}
  \end{figure}
\end{example}

The next example shows that different occurrences of $\pi$ in
$\sigma$ do not necessarily lead to different subgraphs of
$G_p(\sigma)$.

\begin{example}
    If $p=12$, $\pi=312$ and $\sigma=3412$ there are two occurrences of
    $\pi$ in $\sigma$. The index injections are $i=134$ and $i'=234$.
    However, as $(i_2,i_3) = (i'_2,i'_3)$ and the fact that $\{2,3\}$ is
    the only index set of $p$ in $\pi$, we obtain the same
    injection $\Phi_p$ and therefore the same subgraph of $G_p(\sigma)$
    for both index injections.
\end{example}

We call a property of a graph $G$ \emph{hereditary} if it is invariant under
isomorphisms and for every subgraph of $G$ the property also holds. For example
the properties of being a forest, bipartite, planar or $k$-colorable are
hereditary properties, while being a tree is not hereditary. A set of graphs
defined by a hereditary property is a \emph{hereditary class}.

Given $c$, a property of graphs, we define a set of permutations:
\[
    \ms{G}_{p,c} = \left\{ \pi \in \mf{S} \colon G_p(\pi) \text{ has property } c \right\}.
\]

\begin{theorem}\label{thm:classical-subsets}
    Let $c$ be a hereditary property of graphs. For any pattern $p$ the set
    $\ms{G}_{p,c}$ is a permutation class, i.e.,~there is a set of classical
    permutation patterns $M$ such that
    \[
      \ms{G}_{p,c} = \Av(M).
    \]
\end{theorem}

\begin{proof}
  Let $\sigma$ be a permutation such that $G_p(\sigma)$ satisfies
  the hereditary property $c$ and let $\pi$ be a pattern in $\sigma$.
  By Lemma~\ref{prop:permutation-subgraphs} the graph $G_p(\pi)$
  is isomorphic to a subgraph of $G_p(\sigma)$ and thus inherits the
  property $c$.
\end{proof}

In the remainder of this section we consider two hereditary classes of graphs:
bipartite graphs and forests. Recall that a non-empty simple graph on $n$
vertices ($n>0$) is a \emph{tree} if and only if it is connected and has $n-1$
edges. An equivalent condition is that the graph has at least one vertex and no
simple cycles (a simple cycle is a sequence of unique vertices $v_1, \ldots,
v_k$ with edges $v_1v_2, \ldots, v_{k-1}v_k, v_kv_1$). A \emph{forest} is a
disjoint union of trees. The empty graph is a forest but not a tree.
\emph{Bipartite} graphs are graphs that can be colored with two colors in such a
way that no edge joins two vertices with the same color. We note that every
forest is a bipartite graph.\\

Table~\ref{tab:bipartite} shows experimental results, obtained
using~\cite{bisc}, on which occurrence graphs with respect to the patterns $p =
12$, $p = 123$, $p =132$ are bipartite.

\begin{center}
\begin{table}[h]
  \caption{Experimental results for bipartite occurrence graphs. Computed with permutations up to length $8$}
  \label{tab:bipartite}
  \begin{tabular}{| l | l | l | l |}
      \hline
      $p$   & basis                              & Number seq.\                         & OEIS \\
      \hline
      $12$  & $123$, $1432$, $3214$              & $1, 2, 5, 12, 26, 58, 131, 295$      & A116716 \\
      \hline
      $123$ & $1234$, $12543$, $14325$, $32145$  & $1, 2, 6, 23, 100, 462, 2207, 10758$ & \\
      \hline
      $132$ & $1432$, $12354$, $13254$, $13452$, & $1, 2, 6, 23, 95, 394, 1679, 7358$   & \\
            & $15234$, $21354$, $23154$,         &                                      & \\
            & $31254$, $32154$                   &                                      & \\
      \hline
  \end{tabular}
\end{table}
\end{center}

In the following theorem we verify the statements in line $1$ of
Table~\ref{tab:bipartite}. We leave the remainder of the table as conjectures.

\begin{theorem}\label{thm:bipartite}
  Let $c$ be the property of being bipartite and $p=12$. Then
  \[
    \ms{G}_{p,c} = \Av\left( 123, 1432, 3214 \right).
  \]
  The OEIS sequence A116716 enumerates a symmetry of this permutation class.
\end{theorem}

The proof of this theorem relies on a proposition characterizing the cycles
in the graphs under consideration.

\begin{proposition}\label{prop:length-of-cycle-with-p-12}
  If the graph $G_{12}(\pi)$ has a cycle of length $k>4$ then it also has a
  cycle of length $3$.
\end{proposition}

\begin{proof}
  Let $p=12$ and let $\pi$ be a permutation such that $G_p(\pi)$
  contains a cycle of length $k>4$. Label the vertices in the
  cycle $v_1, \ldots, v_k$ with $v_l = \{ i_l, j_l \}$, $i_l < j_l$,
  for $l=1,\ldots,k$.

  The vertices $v_1$ and $v_2$ in the cycle have exactly one index
  in common. If $i_2 = j_1$ then the vertices $v_1$, $v_2$, $\{i_1, j_2\}$ form
  a triangle. So we can assume $i_1 = i_2$. If $j_1 < j_2$ and $\pi(j_1) <
  \pi(j_2)$ (or $j_1 > j_2$ and $\pi(j_1) > \pi(j_2)$) then $u = \{
  j_1, j_2 \}$ is an occurrence of $p$ in $\pi$, forming a triangle
  $v_1,v_2,u$. So either $j_1 > j_2$ and $\pi(j_1) <
  \pi(j_2)$ holds, or, without loss of generality (see
  Figure~\ref{fig:length-of-cycle_v1v2}), $j_1 < j_2$ and $\pi(j_1) > \pi(j_2)$.

  \begin{figure}[htbp]\begin{gather*}
    \begin{tikzpicture}[scale=.4, baseline={([yshift=-3pt]current bounding box.center)}]
        \foreach \x in {2,5,8} {
          \foreach \y in {2,5,8} {
            \draw[gray] (0,\y) -- (10,\y);
            \draw[gray] (\x,0) -- (\x,10);
          }
        }
        \foreach \x in {(2,2),(5,8),(8,5)} {\fill[black] \x circle (5pt);}
        \draw[black] (2,2) -- (5,8) node [midway, fill=white] {$v_1$};
        \draw[black] (2,2) -- (8,5) node [midway, fill=white] {$v_2$};
        \node[below] at  (0.5,-0.2) {$\cdots$};
        \node[below] at  (2,  0) {$i_1$};
        \node[below] at  (3.5,-0.2) {$\cdots$};
        \node[below] at  (5,  0) {$j_1$};
        \node[below] at  (6.5,-0.2) {$\cdots$};
        \node[below] at  (8,  0) {$j_2$};
        \node[below] at  (9.5,-0.2) {$\cdots$};
        \node[right] at (11.2,9.5) {$\vdots$};
        \node[right] at (10,  8) {$\pi(j_1)$};
        \node[right] at (11.2,6.5) {$\vdots$};
        \node[right] at (10,  5) {$\pi(j_2)$};
        \node[right] at (11.2,3.5) {$\vdots$};
        \node[right] at (10,  2) {$\pi(i_1)$};
        \node[right] at (11.2,  1) {$\vdots$};
    \end{tikzpicture}\end{gather*} \caption{The vertices $v_1$ and
      $v_2$ (shown as line segments inside the permutation $\pi$)
      share the index $i_1$}
    \label{fig:length-of-cycle_v1v2}
  \end{figure}
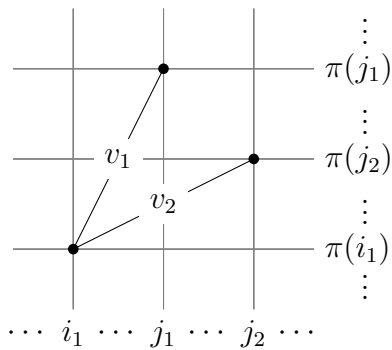

  Next we look at the edge $v_2v_3$ in the cycle. If the vertices
  have the index $i_1$ in common then $v_1,v_2,v_3$ forms a triangle
  in $G_p(\pi)$. So assume that $v_2$ and $v_3$ have the index $j_2$
  in common with the conditions $i_3 > i_1$ and $\pi(i_3) < \pi(i_1)$
  (because else there are more vertices and edges forming a cycle
  of length $3$ in $G_p(\pi)$). Continuing down this road we
  know that $v_3v_4$ is an edge with shared index $i_3$ and
  conditions $j_3 > i_3$ and $\pi(j_3) < \pi(j_1)$, see
  Figure~\ref{fig:length-of-cycle_v1v2v3v4}, where we consider the case $i_3>j_1$,
  and $\pi(j_3) < \pi(i_1)$.

  \begin{figure}[htbp]\begin{gather*}
    \begin{tikzpicture}[scale=.4, baseline={([yshift=-3pt]current bounding box.center)}]
        \foreach \x in {2,5,8,11,14} {
          \foreach \y in {2,5,8,11,14} {
            \draw[gray] (0,\y) -- (16,\y);
            \draw[gray] (\x,0) -- (\x,16);
          }
        }
        \foreach \x in {(2,8),(5,14),(11,11),(8,2),(14,5)} {\fill[black] \x circle (5pt);}
        \draw[black] ( 2, 8) -- ( 5,14) node [midway, fill=white] {$v_1$};
        \draw[black] ( 2, 8) -- (11,11) node [midway, fill=white] {$v_2$};
        \draw[black] ( 8, 2) -- (11,11) node [midway, fill=white] {$v_3$};
        \draw[black] ( 8, 2) -- (14, 5) node [midway, fill=white] {$v_4$};
        \node[below] at  (0.5,0) {$\cdots$};
        \node[below] at  (2  ,0) {$i_1$};
        \node[below] at  (3.5,0) {$\cdots$};
        \node[below] at  (5  ,0) {$j_1$};
        \node[below] at  (6.5,0) {$\cdots$};
        \node[below] at  (8  ,0) {$i_3$};
        \node[below] at  (9.5,0) {$\cdots$};
        \node[below] at (11  ,0) {$j_2$};
        \node[below] at (12.5,0) {$\cdots$};
        \node[below] at (14  ,0) {$j_3$};
        \node[below] at (15.5,0) {$\cdots$};
        \node[right] at (16,15.5) {$\vdots$};
        \node[right] at (16,  14) {$\pi(j_1)$};
        \node[right] at (16,12.5) {$\vdots$};
        \node[right] at (16,  11) {$\pi(j_2)$};
        \node[right] at (16, 9.5) {$\vdots$};
        \node[right] at (16,   8) {$\pi(i_1)$};
        \node[right] at (16, 6.5) {$\vdots$};
        \node[right] at (16,   5) {$\pi(j_3)$};
        \node[right] at (16, 3.5) {$\vdots$};
        \node[right] at (16,   2) {$\pi(i_3)$};
        \node[right] at (16,   1) {$\vdots$};
    \end{tikzpicture}\end{gather*}
      \caption{The vertices $v_1, v_2, v_3, v_4$}
    \label{fig:length-of-cycle_v1v2v3v4}
  \end{figure}
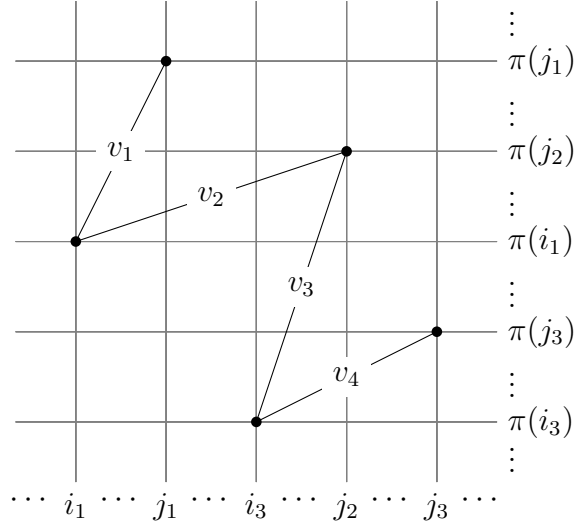

  Graphically, it is quite obvious that we cannot extend the graphical line path
  in Figure~\ref{fig:length-of-cycle_v1v2v3v4} with more southwest-northeast
  line segments (a sequence of vertices $v_5, \ldots, v_k$) such that the
  extension closes the path into a cycle without adding more edges (line
  segments) between vertices that are not adjacent in the cycle and thus forming
  a cycle of length $3$ in the occurrence graph. More precisely for an
  edge between $v_k$ and $v_1$ to exist we must have $v_k = \{i_k, j_k\}$ with
  a non-empty intersection with $v_1$. Analyzing each of these cases completes
  the proof.
\end{proof}

\begin{proof}[Proof of Theorem~\ref{thm:bipartite}]
  If $\pi$ contains any of the patterns $123$, $1432$, $3214$ then $G_p(\pi)$
  contains a subgraph that is isomorphic to a triangle. So if
  $\pi \notin \Av\left(123, 1432, 3214 \right)$ then $G_p(\pi)$ contains an odd
  cycle and is therefore not bipartite.

  On the other hand, let $\pi$ be a permutation such that $G_p(\pi)$ is not
  bipartite. Then the occurrence graph contains an odd cycle which by
  Proposition~\ref{prop:length-of-cycle-with-p-12} implies the graph has a
  cycle of length $3$. The indices corresponding to this cycle form a pattern of
  length $3$ or $4$ in $\pi$ with occurrence graph that is a cycle of lengths
  $3$. It is easy to see that the only permutations of these length whith
  occurrence graph a cycle of length $3$ are $123$, $1432$ and $3214$. Therefore
  $\pi$ must contain at least one of the patterns. \end{proof}

Table~\ref{tab:forest} considers occurrence graphs that are forests.

\begin{center}
\begin{table}[h]
    \caption{Experimental results for occurrence graphs that are forests. Computed with permutations up to length $8$}
    \label{tab:forest}
  \begin{tabular}{| l | l | l | l |}
      \hline
      $p$   & basis                               & Number seq.\                       & OEIS \\
      \hline
      $12$  & $123$, $1432$, $2143$, $3214$       & $1, 2, 5, 11, 24, 53, 117, 258$    & A052980 \\
      \hline
      $123$ & $1234$, $12543$, $13254$, $14325$,  & $1, 2, 6, 23, 97, 429, 1947, 8959$ & \\
            & $21354$, $21435$, $32145$           &                                    & \\
      \hline
      $132$ & $1432$, $12354$, $12453$, $12534$,  & $1, 2, 6, 23, 90, 359, 1481, 6260$ & \\
            & $13254$, $13452$, $14523$, $15234$, &                                    & \\
            & $21354$, $21453$, $21534$, $23154$, &                                    & \\
            & $31254$, $32154$                    &                                    & \\
      \hline
  \end{tabular}
\end{table}
\end{center}

In the following theorem we verify the statements in line $1$ of
Table~\ref{tab:forest}. We leave the remainder of the table as conjectures.

\begin{theorem}\label{thm:forest}
  Let $c$ be the property of being a forest and $p=12$. Then \[
  \ms{G}_{p,c} = \Av\left( 123, 1432, 2143, 3214 \right) .\]
\end{theorem}

\begin{proof}
  If $\pi$ contains the pattern $2143$ then $G_p(\pi)$ contains a
  subgraph that is isomorphic to a cycle of length four, according to
  Lemma~\ref{prop:permutation-subgraphs}, because $G_p(2143)$ is a
  cycle of length four. If $\pi$ contains any of the patterns $123$,
  $1432$, $3214$ then its occurrence graph is bipartite by
  Theorem~\ref{thm:bipartite}, and in particular not a forest.

  On the other hand, let $\pi$ be a permutation such that $G_p(\pi)$
  is not a forest. Then the occurrence graph contains a cycle.
  Proposition~\ref{prop:length-of-cycle-with-p-12} implies that the graph must
  have length either $3$ or $4$. But it's easy
  to see that the only permutations with occurrence graphs that are
  cycles of length $3$ or $4$ are $123,1432,2143,3214$. Therefore
  $\pi$ must contain at least one of the patterns.
\end{proof}

\section{Non-hereditary properties of graphs}
\label{sec:non-hereditary}

This section is devoted to graph properties that are not hereditary. Thus
Theorem~\ref{thm:classical-subsets} does not guarantee that the permutations
whose occurrence graphs satisfy the property form a pattern class.
Experimental results in Table~\ref{tab:non-hereditary} seem to suggest that some
properties still give rise to permutation classes.

\begin{center}
\begin{table}[h]
  \caption{$p = 12$. Computed with permutations up to length $8$}
  \label{tab:non-hereditary}
  \begin{tabular}{| l | l | l | l |}
      \hline
      Property  & basis                           & Number seq.\                         & OEIS \\
      \hline
      connected & see Figure~\ref{fig:meshp}      & $1, 2, 6, 23, 111, 660, 4656, 37745$ & \\
      \hline
      chordal   & $1234$, $1243$, $1324$, $2134$, & $1, 2, 6, 19, 61, 196, 630, 2025$    & \\
                & $2143$                          &                                      & \\
      \hline
      clique    & $1234$, $1243$, $1324$, $1342$, & $1, 2, 6, 12, 20, 30, 42, 56$        & A002378\footnote{From from $n=2$.} \\
                & $1423$, $2134$, $2143$, $2314$, &                                      & \\
                & $2413$, $3124$, $3142$, $3412$  &                                      & \\
      \hline
      tree      & very large non-classical basis    & $0, 1, 4, 9, 16, 25, 36, 49$         & A000290 \\
      \hline
  \end{tabular}
\end{table}
\end{center}

\begin{figure}[htbp]
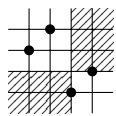

  \mpattern{scale=\pattdispscale}{ 4 }{ 1/3, 2/4, 3/1, 4/2 }{0/0,0/1, 1/0,1/1, 2/0,2/1, 3/2,3/3,3/4, 4/2,4/3,4/4}
  \caption{Mesh pattern classifying connected occurrence graphs with respect to $p = 12$}\label{fig:meshp}
\end{figure}

\begin{theorem}\label{thm:connected}
  Let $c$ be the property of being connected and $p=12$. Then
  \[
    \ms{G}_{p,c} = \Av\left( m \right),
  \]
  where $m$ is the mesh pattern in Figure~\ref{fig:meshp}. The generating function
  for the enumeration of these permutations is
  \[
    \frac{F(x)-x}{(1-x)^2} + \frac{1}{1-x}
  \]
  where $F(x) = 1 - 1/\sum{k! x^k}$ is the generating function for the skew-indecomposable permutations,
  see e.g., Comtet~\cite[p.~261]{comtet}
\end{theorem}

\begin{proof}
    If the graph $G_p(\pi)$ is disconnected then $\pi$ has two occurrences
    of $p = 12$ in distinct skew-components, $A$ and $B$, which we can take to
    be consecutive in the skew-decomposition of $\pi = \cdots \ominus A \ominus B \ominus \cdots$.
    Let $ab$ be any
    occurrence of $12$ in $A$. Let $u$ be the highest point in $B$ and $v$ be
    the leftmost point in $B$. Then $abuv$ is an occurrence of the mesh pattern.
    It is clear that an occurrence $abuv$ of the mesh pattern will correspond to
    two vertices $ab$, $uv$ in the occurrence graph, and the shadings ensure
    that there is no path between them.

    The enumeration follows from the fact that these permutations must have
    no, or exactly one, skew-component of size greater than $1$. The first case
    is counted by $1/(1-x)$ while the second case is counted by $(F(x)-x)/(1-x)^2$.
\end{proof}

Note that our software suggests a very large non-classical basis for the
permutations with a tree as an occurrence graph. We omit displaying this basis
here. However, since a graph is a tree if and only if it is a non-empty
connected forest we obtain:

\begin{corollary}
    Let $c$ be the property of being a tree and $p=12$. Then
    \[
      \ms{G}_{p,c} = \Av\left( 123, 1432, 2143, 3214, m \right) \setminus \Av(12),
    \]
    where $m$ is the mesh pattern in Figure~\ref{fig:meshp}.
\end{corollary}

\begin{proof}
    This follows from Theorems~\ref{thm:forest} and~\ref{thm:connected}. We must
    remove the decreasing permutations since they all have empty occurrence
    graphs.
\end{proof}

We end with proving the enumeration for the permutations in the corollary above.
The proof is a rather tedious, but simple, induction proof.

\begin{theorem}\label{thm:nr-of-12-trees}
  The number of permutations of length $n$ in $\ms{G}_{12,\text{tree}}$ is $(n-1)^2$.
\end{theorem}

\section{Future work}
\label{future-work}

We expect the conjectures in lines~2 and~3 in Tables~\ref{tab:bipartite}
and~\ref{tab:forest} to follow from an analysis of the cycle structure of
occurrence graphs with respect to the patterns $123$ and $132$, similar to what
we did in Proposition~\ref{prop:length-of-cycle-with-p-12} for the pattern $12$.

Other natural hereditary graph properties to consider would be $k$-colorable
graphs, for $k > 2$, as these are supersets of bipartite graphs. Also
planar graphs, which lie between forests and $4$-colorable graphs.

It might also be interesting to consider the intersection $\bigcap_{p \in M}
\ms{G}_{p,c}$ where $M$ is some set of patterns, perhaps all.

We would like to note that Jason Smith~\cite{jsmith} independently defined
occurrence graphs and used them to prove a result on the shellability of a large
class of intervals of permutations.

\appendix
\section{Proof of Theorem~\ref{thm:nr-of-12-trees}}
\label{sec:extraproofs}

We start by introducing a new notation.

\begin{definition}\label{defn:prefix-permutation}
 Let $\pi\in\mf{S}_n$ and $k$ be an integer such that $1 \leq k
 \leq n+1$. The \emph{$k$-prefix of $\pi$} is the permutation
 $\pi'\in\mf{S}_{n+1}$ defined by $\pi'(1) = k$ and \[ \pi'(i+1) =
 \left\{\begin{array}{ll} \pi(i) &\mbox{if } \pi(i) < k, \\ \pi(i)+1
 &\mbox{if } \pi(i) \geq k \end{array}\right. \] for $i=1,\ldots,n$.
 We denote $\pi'$ with $k\succ\pi$. In a similar way we define
 the \emph{$k$-postfix of $\pi$} as the permutation $\pi\prec k$ in
 $\mf{S}_{n+1}$.
\end{definition}

\begin{example}\label{ex:prefix-permutation}
 Let $\pi=42135$ and $k=2$. Visually, if we draw the grid
 representation of $\pi$, we are putting the new number $k$ to the
 left on the $x$-axis and raising all the numbers $\geq k$ on the
 $y$-axis by one. Thus we have $2\succ 42135 = 253146$. See Figure~\ref{fig:2-prefix-of-42135}
\end{example}

\begin{figure}[htbp]
 \begin{gather*}
   \begin{tikzpicture}[scale=.4, baseline={([yshift=-3pt]current bounding box.center)}]
     \def \n {5}
     \foreach \x in {1,...,\n} {
       \draw[gray] (0,\x) -- (\n+1,\x);
       \draw[gray] (\x,0) -- (\x,\n+1);
     }
     \foreach \x in {(1,4),(2,2),(3,1),(4,3),(5,5)} {\fill[black] \x circle (5pt);}
     \fill[gray] (0,1.5) circle (7pt);
   \end{tikzpicture}
   \;\longmapsto\;
   \begin{tikzpicture}[scale=.4, baseline={([yshift=-3pt]current bounding box.center)}]
     \def \n {6}
     \foreach \x in {1,...,\n} {
       \draw[gray] (0,\x) -- (\n+1,\x);
       \draw[gray] (\x,0) -- (\x,\n+1);
     }
     \foreach \x in {(2,5),(3,3),(4,1),(5,4),(6,6)} {\fill[black] \x circle (5pt);}
     \fill[gray] (1,2) circle (7pt);
   \end{tikzpicture}
   \end{gather*}
 \caption{The $2$-prefix of $42135$ is $253146$}\label{fig:2-prefix-of-42135}
\end{figure}
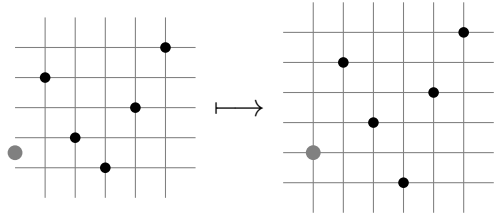

%

We note that for every permutation $\pi'\in\mf{S}_{n+1}$ there is
one and only one pair $(k,\pi)$ such that $\pi' = k\succ\pi$. We let
$k=\pi'(1)$ and $\pi = \st(\pi'(2) \cdots \pi'(n+1))$.

\begin{proof}[Proof of~\ref{thm:nr-of-12-trees}]
  Let $p=12$. We start by considering three base-cases.

  For $n=1$ the occurrence graph is the empty graph. For $n=2$ we
  get two occurrence graphs: $G_p(12)$ is a single node graph and
  $G_p(21)$ is the empty graph. For $n=3$ we have $3!=6$ different
  permutations $\pi$. Of those we calculate that $132$, $213$, $231$
  and $312$ result in connected occurrence graphs on one or two nodes
  but $G_p(123)$ is a triangle and $G_p(321)$ is the empty graph.

  We have thus showed that the claimed enumeration is true for
  $n=1,2,3$.

  For the inductive step we assume $n \geq 4$ and let $\pi$ be a
  permutation of length $n$. We look at four different cases of $k$
  to construct $\pi' = k\succ\pi$. We let $x$, $y$ and $z$ be the
  indices of $n-1$, $n$ and $n+1$ in $\pi'$ respectively.

  \begin{enumerate}[(I)]
    \item \underline{$k\leq n-2$}:
      The index sets $\{1,x\}$, $\{1,y\}$ and $\{1,z\}$ of $p$ in
      $\pi'$ all share exactly one common element and thus form a
      triangle in $G_p(\pi')$. Therefore there are no permutations
      $\pi$ resulting in the occurrence graph $G_p(\pi')$ being a
      tree.

    \item \underline{$k=n-1$}:
      Let $T(n+1)$ denote the number of permutations $\pi'$ of
      length $n+1$ with $\pi'(1)=n-1$ such that $G_p(\pi')$ is a
      tree. Note that $T(1)=T(2)=0$, $T(3)=1$ and $T(4)=2$. In order
      to obtain a formula for $T$ we need to look at a few subcases:

      \begin{enumerate}[i)]
        \item If $y<z$ then $\{1,y\}$, $\{1,z\}$ and $\{y,z\}$ form
          an triangle in $G_p(\pi')$, see Figure~\ref
          {fig:tree-proof-(n-1)-y_le_z}. Independent of the
          permutation $\pi$, the graph $G_p(\pi')$ is not a tree.

          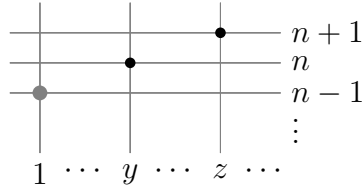
\begin{figure}[htbp]\begin{gather*}
            \begin{tikzpicture}[scale=.4, baseline={([yshift=-3pt]current bounding box.center)}]
                \foreach \x in {1,4,7} {
                  \foreach \y in {2,3,4} {
                    \draw[gray] (0,\y) -- (9,\y);
                    \draw[gray] (\x,0) -- (\x,5);
                  }
                } \fill[gray] (1,2) circle (7pt);
                \foreach \x in {(4,3),(7,4)} {\fill[black] \x circle (5pt);}
                \node[below] at (1,0) {$1$};
                \node[below] at (2.5,0) {$\cdots$};
                \node[below] at (4,0) {$y$};
                \node[below] at (5.5,0) {$\cdots$};
                \node[below] at (7,0) {$z$};
                \node[below] at (8.5,0) {$\cdots$};
                \node[right] at (9,4) {$n+1$};
                \node[right] at (9,3) {$n$};
                \node[right] at (9,2) {$n-1$};
                \node[right] at (9,1) {$\vdots$};
            \end{tikzpicture}\end{gather*}
            \caption{$k=n-1$ and $y<z$}\label{fig:tree-proof-(n-1)-y_le_z}
          \end{figure}

        \item Assume $y>z$ and $z \neq 2$, see
          Figure~\ref{fig:tree-proof-(n-1)-y_gr_z-z_neq_2}. Then $\pi'(2) < n-1$
          and $\{1,z\}$, $\{2,z\}$, $\{2,y\}$ and $\{1,y\}$ form a
          cycle of length $4$ in $G_p(\pi')$, resulting in it not
          being a tree.

          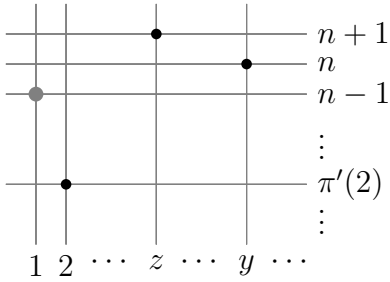
\begin{figure}[htbp]\begin{gather*}
            \begin{tikzpicture}[scale=.4, baseline={([yshift=-3pt]current bounding box.center)}]
                \foreach \x in {1,2,5,8} {
                  \foreach \y in {2,5,6,7} {
                    \draw[gray] (0,\y) -- (10,\y);
                    \draw[gray] (\x,0) -- (\x,8);
                  }
                } \fill[gray] (1,5) circle (7pt);
                \foreach \x in {(2,2),(5,7),(8,6)} {\fill[black] \x circle (5pt);}
                \node[below] at (1,0) {$1$};
                \node[below] at (2,0) {$2$};
                \node[below] at (3.5,0) {$\cdots$};
                \node[below] at (5,0) {$z$};
                \node[below] at (6.5,0) {$\cdots$};
                \node[below] at (8,0) {$y$};
                \node[below] at (9.5,0) {$\cdots$};
                \node[right] at (10,7) {$n+1$};
                \node[right] at (10,6) {$n$};
                \node[right] at (10,5) {$n-1$};
                \node[right] at (10,3.5) {$\vdots$};
                \node[right] at (10,2) {$\pi'(2)$};
                \node[right] at (10,1) {$\vdots$};
            \end{tikzpicture}\end{gather*}
            \caption{$k=n-1$, $y>z$ and $z \neq 2$}\label{fig:tree-proof-(n-1)-y_gr_z-z_neq_2}
          \end{figure}

        \item Now lets assume $y>z$ and $z=2$, see
          Figure~\ref{fig:tree-proof-(n-1)-y_gr_z-z_eq_2}.

          \begin{figure}[htbp]\begin{gather*}
            \begin{tikzpicture}[scale=.4, baseline={([yshift=-3pt]current bounding box.center)}]
                \foreach \x in {1,2,3,6} {
                  \foreach \y in {2,5,6,7} {
                    \draw[gray] (0,\y) -- (8,\y);
                    \draw[gray] (\x,0) -- (\x,8);
                  }
                } \fill[gray] (1,5) circle (7pt);
                \foreach \x in {(2,7),(3,2),(6,6)} {\fill[black] \x circle (5pt);}
                \node[below] at (1,0) {$1$};
                \node[below] at (2,0) {$2$};
                \node[below] at (3,0) {$3$};
                \node[below] at (4.5,0) {$\cdots$};
                \node[below] at (6,0) {$y$};
                \node[below] at (7.5,0) {$\cdots$};
                \node[right] at (8,7) {$n+1$};
                \node[right] at (8,6) {$n$};
                \node[right] at (8,5) {$n-1$};
                \node[right] at (8,3.5) {$\vdots$};
                \node[right] at (8,2) {$\pi'(3)$};
                \node[right] at (8,1) {$\vdots$};
            \end{tikzpicture}\end{gather*}
            \caption{$k=n-1$, $y>z$ and $z=2$}\label{fig:tree-proof-(n-1)-y_gr_z-z_eq_2}
          \end{figure}
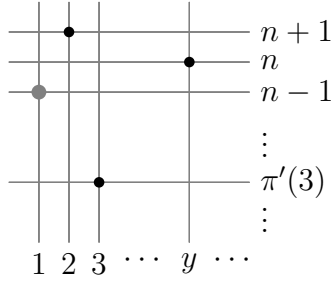

          If $y \geq 5$ then the vertices $\{1,y\}$, $\{3,y\}$ and
          $\{4,y\}$ form a cycle in $G_p(\pi')$.

          If $y = 3$ then $\{1,2\}$ and $\{1,3\}$ will be an
          isolated path component in $G_p(\pi')$, making $\pi' =
          (n-1)(n+1)n(n-2) \cdots 1$ the only permutation such that
          the occurrence graph $G_p(\pi')$ is a tree.

          Now fix $y=4$ and lets look at some subsubcases for the
          value of $\pi'(3)$.

          \begin{enumerate}[a)]
            \item If $\pi'(3) \leq n-4$ then $\pi'(3)n$,
              $\pi'(3)(n-2)$ and $\pi'(3)(n-3)$ are all occurrences
              of $p$ in $\pi'$, with the respective index sets
              forming an triangle in $G_p(\pi')$.
            \item If $\pi'(3) = n-2$ then $\pi' =
              (n-1)(n+1)(n-2)n(n-3) \cdots 1$ is the only
              permutation resulting in $G_p(\pi')$ being a tree.
            \item If $\pi'(3) = n-3$ we look at
              Figure~\ref{fig:tree-proof-(n-1)-y_eq_4-z_eq_2} where the
              permutation $\pi'$ is shown.

              \begin{figure}[htbp]\begin{gather*}
                \begin{tikzpicture}[scale=.4, baseline={([yshift=-3pt]current bounding box.center)}]
                    \foreach \x in {1,2,3,4} {
                      \foreach \y in {2,3,4,5,6} {
                        \draw[gray] (0,\y) -- (6,\y);
                        \draw[gray] (\x,0) -- (\x,7);
                      }
                    } \fill[gray] (1,4) circle (7pt);
                    \foreach \x in {(2,6),(3,2),(4,5)} {\fill[black] \x circle (5pt);}
                    \node[below] at (1,0) {$1$};
                    \node[below] at (2,0) {$2$};
                    \node[below] at (3,0) {$3$};
                    \node[below] at (4,0) {$4$};
                    \node[below] at (5.5,0) {$\cdots$};
                    \node[right] at (6,6) {$n+1$};
                    \node[right] at (6,5) {$n$};
                    \node[right] at (6,4) {$n-1$};
                    \node[right] at (6,3) {$n-2$};
                    \node[right] at (6,2) {$n-3$};
                    \node[right] at (6,1) {$\vdots$};
                \end{tikzpicture}\end{gather*}
                \caption{$k=n-1$, $y=4$ and $z=2$}\label{fig:tree-proof-(n-1)-y_eq_4-z_eq_2}
              \end{figure}
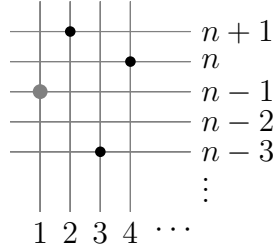

              The permutation $\sigma = \st(\pi'(3)\cdots\pi'(n+1))$
              is just like $\pi'$ in the case $k=n-1$ and $z=2$,
              only the length of $\sigma$ is $n-1$. Because $\{1,2\}$
              is a vertex in $G_p(\sigma)$ the occurrence graph of
              $p$ in $\sigma$ is not the empty graph. Thus it is
              easy to see that $G_p(\pi')$ is a tree if and only if
              $G_p(\sigma)$ is a tree, and according to the
              aforementioned case there are $T(n-1)$ such
              permutations $\sigma$.
          \end{enumerate}

          \noindent Summing up the subsubcases there are a total of
          $1+1+T(n-1)$ permutations $\pi'$ making the occurrence
          graph a tree, i.e.,~$T(n+1)=2+T(n-1)$. Because $T(4)=2$
          and $T(3)=1$ we deduce that $T(n+1)=n-1$.
      \end{enumerate}

      \noindent The whole case $k=n-1$ gives us that there are $n-1$
      permutation $\pi'$ such that $G_p(\pi')$ is a tree.

    \item \underline{$k=n$}:
      We need to examine three subcases:

      \begin{enumerate}[i)]
        \item If $z \geq 4$ then $\{1,z\}$, $\{2,z\}$, $\{3,z\}$ are
          all index sets of $p$ in $\pi'$, forming a triangle in
          $G_p(\pi')$.

        \item If $z = 3$, then $\{1,2\}$ is an index set of $p$ in
          $\pi$ making the occurrence graph $G_p(\pi)$ non-empty,
          see Figure~\ref{fig:tree-proof-n-z_eq_3}.

          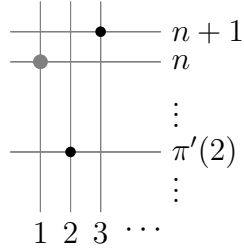
\begin{figure}[htbp]\begin{gather*}
            \begin{tikzpicture}[scale=.4, baseline={([yshift=-3pt]current bounding box.center)}]
                \foreach \x in {1,2,3} {
                  \foreach \y in {2,5,6} {
                    \draw[gray] (0,\y) -- (5,\y);
                    \draw[gray] (\x,0) -- (\x,7);
                  }
                } \fill[gray] (1,5) circle (7pt);
                \foreach \x in {(2,2),(3,6)} {\fill[black] \x circle (5pt);}
                \node[below] at (1,0) {$1$};
                \node[below] at (2,0) {$2$};
                \node[below] at (3,0) {$3$};
                \node[below] at (4.5,0) {$\cdots$};
                \node[right] at (5,6) {$n+1$};
                \node[right] at (5,5) {$n$};
                \node[right] at (5,3.5) {$\vdots$};
                \node[right] at (5,2) {$\pi'(2)$};
                \node[right] at (5,1) {$\vdots$};
            \end{tikzpicture}\end{gather*}
            \caption{$k=n$ and $z=3$}\label{fig:tree-proof-n-z_eq_3}
          \end{figure}

          If $\pi'(2) \leq n-3$ then $\pi'(2)(n+1)$, $\pi'(2)(n-1)$
          and $\pi'(2)(n-2)$ are all occcurrences of $p$ in $\pi'$,
          resulting in $G_p(\pi')$ having a triangle.

          If $\pi'(2) = n-1$ then $\{1,3\}$ and $\{2,3\}$ is an
          isolated path component in $G_p(\pi')$ and $\pi' =
          n(n-1)(n+1)(n-2) \cdots 1$ is the only permutation such
          that the occurrence graph is a tree.

          We therefore assume $\pi'(2) = n-2$, see Figure~\ref
          {fig:tree-proof-n-z_eq_3-pi2_eq_n-2}.

          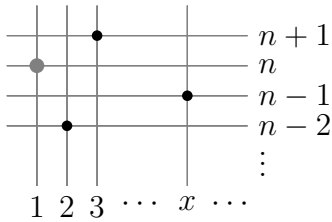
\begin{figure}[htbp]\begin{gather*}
            \begin{tikzpicture}[scale=.4, baseline={([yshift=-3pt]current bounding box.center)}]
                \foreach \x in {1,2,3,6} {
                  \foreach \y in {2,3,4,5} {
                    \draw[gray] (0,\y) -- (8,\y);
                    \draw[gray] (\x,0) -- (\x,6);
                  }
                } \fill[gray] (1,4) circle (7pt);
                \foreach \x in {(2,2),(3,5),(6,3)} {\fill[black] \x circle (5pt);}
                \node[below] at (1,0) {$1$};
                \node[below] at (2,0) {$2$};
                \node[below] at (3,0) {$3$};
                \node[below] at (4.5,0) {$\cdots$};
                \node[below] at (6,0) {$x$};
                \node[below] at (7.5,0) {$\cdots$};
                \node[right] at (8,5) {$n+1$};
                \node[right] at (8,4) {$n$};
                \node[right] at (8,3) {$n-1$};
                \node[right] at (8,2) {$n-2$};
                \node[right] at (8,1) {$\vdots$};
            \end{tikzpicture}\end{gather*}
            \caption{$k=n$, $z=3$ and $\pi'(2)=n-2$}\label{fig:tree-proof-n-z_eq_3-pi2_eq_n-2}
          \end{figure}

          Let $\sigma = \st(\pi'(2) \cdots \pi'(n+1))$. Note that
          the occurrence graphs $G_p(\pi')$ and $G_p(\sigma)$ are
          the same except the former has an the extra vertex
          $\{1,2\}$ and an edge connecting it to a graph
          corresponding to $G_p(\sigma)$. Therefore, $G_p(\pi')$ is
          a tree if and only if $G_p(\sigma)$ is a tree.

          Note that $\sigma(1)=n-2$ and $\sigma(2)=n$ and therefore
          $\sigma$ is like $\pi'$ in the case $k=n-1$ and $z=2$ as
          in Figure~\ref{fig:tree-proof-(n-1)-y_gr_z-z_eq_2}, only
          of length $n$ instead of $n+1$. By the same reasoning as
          in that case the number of permutations $\sigma$ (and
          therefore $\pi'$) such that $G_p(\pi')$ is a tree is
          $T(n)=n-2$.

        \item If $z = 2$, then $\{1,2\}$ is an isolated vertex in
          $G_p(\pi')$, see Figure~\ref{fig:tree-proof-n-z_eq_2}. The
          occurrence graph of $p$ in $\pi'$ is a tree if and only if
          $G_p(\pi)$ is the empty graph which is true if and only if
          $\pi$ is the decreasing permutation. Therefore there is
          only one permutation $\pi' = n(n+1)(n-1) \ldots 1$ such
          that $G_p(\pi')$ is a tree.

          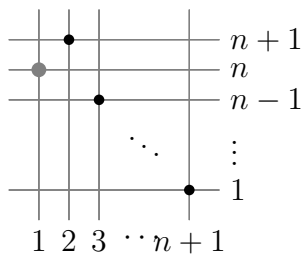
\begin{figure}[htbp]\begin{gather*}
            \begin{tikzpicture}[scale=.4, baseline={([yshift=-3pt]current bounding box.center)}]
                \foreach \x in {1,2,3,6} {
                  \foreach \y in {1,4,5,6} {
                    \draw[gray] (0,\y) -- (7,\y);
                    \draw[gray] (\x,0) -- (\x,7);
                  }
                } \fill[gray] (1,5) circle (7pt);
                \foreach \x in {(2,6),(3,4),(6,1)} {\fill[black] \x circle (5pt);}
                \node[below] at (1,0) {$1$};
                \node[below] at (2,0) {$2$};
                \node[below] at (3,0) {$3$};
                \node[below] at (4.5,0) {$\cdots$};
                \node[below] at (6,0) {$n+1$};
                \node        at (4.5,2.7) {$\ddots$};
                \node[right] at (7,6) {$n+1$};
                \node[right] at (7,5) {$n$};
                \node[right] at (7,4) {$n-1$};
                \node[right] at (7,2.5) {$\vdots$};
                \node[right] at (7,1) {$1$};
            \end{tikzpicture}\end{gather*}
            \caption{$k=n$ and $z=2$}\label{fig:tree-proof-n-z_eq_2}
          \end{figure}
      \end{enumerate}

      \noindent To sum up the the case $k=n$ there are $1+(n-2)+1 =
      n$ permutation $\pi'$ such that $G_p(\pi')$ is a tree.

    \item \underline{$k=n+1$}:
      Every occurrence $\pi(i)\pi(j)$ of $p$ in $\pi$ is also an
      occurrence of $p$ in $\pi'$, but with index set $\{ i+1, j+1
      \}$ instead of $\{ i, j \}$. There are no more occurrences of
      $p$ in $\pi'$ because $\pi'(1)=n+1 > \pi'(j')$ for every $j'>1$
      so $\pi'(1)\pi'(j')$ is not an occurrence of $p$ for any $j'>1$.

      This means that $G_{12}(\pi') \cong G_{12}(\pi)$ so by the
      induction hypothesis we obtain that there are $(n-1)^2$
      permutations $\pi'$ such that the occurrence graph is a tree
      for this value of $k$.

  \end{enumerate}

  \noindent To sum up the four instances there is a total of $0 +
  (n-1) + n + (n-1)^2 = n^2$ permutations $\pi'$ such that
  $G_p(\pi')$ is a tree.
\end{proof}

\bibliographystyle{alpha}
\bibliography{occgraphs}

\end{document}